\documentclass[reqno,centertags]{amsart}
\usepackage{amsmath}
\usepackage{amssymb}
\usepackage{amsfonts}

\newtheorem{theorem}{Theorem}
\theoremstyle{plain}

\newtheorem{lemma}{Lemma}

\newtheorem{problem}{Problem}

\numberwithin{equation}{section}

\input{tcilatex}

\begin{document}
\title[]{Ruled and quadric surfaces of finite Chen-type}
\author{Hassan Al-Zoubi }
\address{Department of Mathematics, Al-Zaytoonah University of Jordan, P.O.
Box 130, Amman, Jordan 11733}
\email{dr.hassanz@zuj.edu.jo}
\author{Stylianos Stamatakis}
\address{Department of Mathematics, Aristotle University of Thessaloniki}
\email{stamata@math.auth.gr}
\author{Hani Almimi }
\address{Department of Computer Science, Al-Zaytoonah University of Jordan, P.O.
Box 130, Amman, Jordan 11733}
\email{Hani.Mimi@zuj.edu.jo}
\date{}

\begin{abstract}
In this paper, we study ruled surfaces and quadrics in the 3-dimensional
Euclidean space which are of finite $III$-type, that is, they are of finite
type, in the sense of B.-Y. Chen, with respect to the third fundamental
form. We show that helicoids and spheres are the only ruled and quadric
surfaces of finite $III$-type, respectively.
\end{abstract}

\maketitle

\section{Introduction}
Let $M^{n}$ be a (connected) submanifold in the m-dimensional Euclidean
space $E^{m}$. Let $\boldsymbol{x},\boldsymbol{H}$ be the position vector
field and the mean curvature field of $M^{n}$ respectively. Denote by $%
\Delta ^{I}$ the second Beltrami-Laplace operator corresponding to the first
fundamental form $I$ of $M^{n}$. Then, it is well known that \cite{R3}
\begin{equation}
\Delta ^{I}\boldsymbol{x}=-n\boldsymbol{H}.  \notag
\end{equation}

From this formula one can see that $M^{n}$ is a minimal submanifold if and
only if all coordinate functions, restricted to $M^{n}$, are eigenfunctions
of $\Delta ^{I}$ with eigenvalue $\lambda =0$. Moreover in \cite{R15} T.
Takahashi showed that the submanifold $M^{n}$ for which $\Delta ^{I}%
\boldsymbol{x}=\lambda \boldsymbol{x}$, i.e., for which all coordinate
functions are eigenfunctions of $\Delta ^{I}$ with the same eigenvalue $%
\lambda \in \mathbb{Re}$, are precisely either the minimal submanifold with
eigenvalue $\lambda =0$ or the minimal submanifold of hyperspheres $S^{m-1}$
with eigenvalue $\lambda >0$. Although the class of finite type submanifolds
in an arbitrary dimensional Euclidean spaces is very large, very little is
known about surfaces of finite type in the Euclidean 3-space $E^{3}$.
Actually, so far, the only known surfaces of finite type corresponding to
the first fundamental form in the Euclidean 3-space are the minimal
surfaces, the circular cylinders and the spheres. So in \cite{R5} B.-Y. Chen
mentions the following problem

\begin{problem}
\label{(1)}Determine all surfaces of finite Chen $I$-type in $E^{3}$.
\end{problem}

In order to provide an answer to the above problem, important families of
surfaces were studied by different authors by proving that finite type ruled
surfaces, finite type quadrics, finite type tubes, finite type cyclides of
Dupin and finite type spiral surfaces are surfaces of the only known
examples in $E^{3}$. However, for another classical families of surfaces,
such as surfaces of revolution, translation surfaces as well as helicoidal
surfaces, the classification of its finite type surfaces is not known yet.
For a more details, the reader can refer to \cite{R6}.

In this context, Chen and Piccini in \cite{C9}, introduced in the same way the theory of submanifolds of finite type Gauss map. A special case for $E^{3}$ one can ask

\begin{problem}
Classify all surfaces in $E^{3}$ with finite type Gauss map.
\end{problem}

Results concerning this problem can be found in (\cite{A2}, \cite{B1}, \cite{B2}, \cite{B3}).

Later in \cite{G1} O. Garay generalized T. Takahashi's condition studied
surfaces in $E^{3}$ for which all coordinate functions $\left(
x_{1},x_{2},x_{3}\right)$ of $\boldsymbol{x}$ satisfy $\Delta^{I}\boldsymbol{%
x}_{i} = \lambda_{i}x_{i}, i = 1,2,3$, not necessarily with the same
eigenvalue. Another generalization was studied in \cite{D3} for which
surfaces in $E^{3}$ satisfy the condition $\Delta^{I}\boldsymbol{x}= A%
\boldsymbol{x} + B$ $(\ddag)$ where $A \in\mathbb{Re}^{3\times3} ;B \in
\mathbb{Re}^{3\times1}$. It was shown that a surface $S$ in $E^{3}$
satisfies $(\ddag)$ if and only if it is an open part of a minimal surface,
a sphere, or a circular cylinder. Surfaces satisfying $(\ddag)$ are said to
be of coordinate finite type.

In the thematic circle of the surfaces of finite type in the Euclidean space
$E^{3}$, S. Stamatakis and H. Al-Zoubi in \cite{R13} restored attention to
this theme by introducing the notion of surfaces of finite type
corresponding to the second or the third fundamental forms of $S$ in the
following way:

A surface $S$ is said to be of finite type corresponding to the fundamental
form $J$, or briefly of finite $J$-type, $J=II,III$, if the position vector $%
\boldsymbol{x}$ of $S$ can be written as a finite sum of nonconstant
eigenvectors of the operator $\Delta ^{J}$, that is if
\begin{equation}
\boldsymbol{x}=\boldsymbol{x}_{0}+\sum_{i=1}^{k}\boldsymbol{x}_{i},\ \ \ \ \
\ \Delta ^{J}\boldsymbol{x}_{i}=\lambda _{i}\boldsymbol{x}_{i},\ \ \
i=1,...,k,  \label{0}
\end{equation}%
where $\boldsymbol{x}_{0}$ is a fixed vector and $\boldsymbol{x}_{1},...,%
\boldsymbol{x}_{k}$ are nonconstant maps such that $\Delta ^{J}\boldsymbol{\
x}_{i}=\lambda _{i}\boldsymbol{x}_{i},i=1,...,k$. If, in particular, all
eigenvalues $\lambda _{1},\lambda _{2},...,\lambda _{k}$ are mutually
distinct, then $S$ is said to be of $J$-type $k$, otherwise $S$ is said to
be of infinite type. When $\lambda _{i}=0$ for some \emph{i} = 1,..., \emph{k%
}, then $S$ is said to be of null $J$-type $k$.

In general when $S$ is of finite type $k$, it follows from (\ref{0}) that
there exist a monic polynomial, say $R(x)\neq 0,$ such that $R(\Delta ^{J})(%
\boldsymbol{x}-\boldsymbol{c})=\mathbf{0}.$ Suppose that $R(x)=x^{k}+\sigma
_{1}x^{k-1}+...+\sigma _{k-1}x+\sigma _{k},$ then coefficients $\sigma _{i}$%
\ are given by

\begin{eqnarray}
\sigma _{1} &=&-(\lambda _{1}+\lambda _{2}+...+\lambda _{k}),  \notag \\
\sigma _{2} &=&(\lambda _{1}\lambda _{2}+\lambda _{1}\lambda_{3}+...+\lambda
_{1}\lambda _{k}+\lambda _{2}\lambda _{3}+...+\lambda _{2}\lambda
_{k}+...+\lambda _{k-1}\lambda _{k}),  \notag \\
\sigma _{3} &=&-(\lambda _{1}\lambda _{2}\lambda _{3}+...+\lambda
_{k-2}\lambda _{k-1}\lambda _{k}),  \notag \\
&&.............................................  \notag \\
\sigma _{k} &=&(-1)^{k}\lambda _{1}\lambda _{2}...\lambda _{k}.  \notag
\end{eqnarray}

Therefore the position vector $\boldsymbol{x}$ satisfies the following
equation, (see \cite{R3})

\begin{equation}  \label{1}
(\Delta^{J})^{k}\boldsymbol{x}+\sigma_{1}(\Delta^{J})^{k-1}\boldsymbol{x}%
+...+\sigma_{k}( \boldsymbol{x}-\boldsymbol{c})=\boldsymbol{0}.
\end{equation}

In \cite{A4} Ruled surfaces were studied regarding the second fundamental form, another classes of surfaces were investigated in (\cite{A5}, \cite{A9}, \cite{A12}), meanwhile similar study was done but for the Gauss map of the surface as one can see in \cite{A6}. In this paper we contribute to the solution of this problem by investigating
the ruled surfaces and the quadric surfaces in $\mathbb{E}^{3}$. On the other hand it is also interesting studying surfaces in the
three-dimensional Euclidean space of coordinate finite type or coordinate finite type Gauss map with respect to the second or third fundamental form. result concerning this can be found in (\cite{A1}, \cite{A3}, \cite{A8}, \cite{A13}, \cite{A14}, \cite{A15}, \cite{A17}, \cite{S1},\cite{R14}). Our main
results are the following

\begin{theorem}
\label{T1.1} The only ruled surfaces of finite $III$-type in the
three-dimensional Euclidean space are the helicoids.
\end{theorem}

\begin{theorem}
\label{T1.2} The only quadric surfaces of finite $III$-type in the
three-dimensional Euclidean space are the spheres.
\end{theorem}

\section{Proof of Theorem \protect\ref{T1.1}}

In the three-dimensional Euclidean space $\mathbb{E}^{3}$ let $S$ be a ruled
$C^{r}$-surface, $r\geq3$, of nonvanishing Gaussian curvature defined by an
injective $C^{r} $-immersion $\boldsymbol{x}=\boldsymbol{x}(s,t)$ on a
region $U: = I \times \mathbb{R} \,\,\,(I\subset \mathbb{R}$ open interval)
of $\mathbb{R}^{2}$.\footnote{%
The reader is referred to \cite{R11a} for definitions and formulae on ruled
surfaces.} The surface $S$ can be expressed in terms of a directrix curve $%
\varGamma \colon \boldsymbol{\sigma } = \boldsymbol{\sigma }(s)$ and a unit
vector field $\boldsymbol{\rho }(s)$ pointing along the rulings as follows
\begin{equation}  \label{2.1}
S \colon \boldsymbol{x}(s,t) = \boldsymbol{\sigma }(s) + t \, \boldsymbol{%
\rho }(s), \quad s\in I, t \in \mathbb{R}.
\end{equation}
Moreover, we can take the parameter $s$ to be the arc length along the
spherical curve $\boldsymbol{\rho }(s)$. Then we have
\begin{equation*}
\langle \boldsymbol{\sigma}^{\prime },\boldsymbol{\rho }\rangle = 0, \quad
\langle \boldsymbol{\rho },\boldsymbol{\rho }\rangle = 1,\quad \langle
\boldsymbol{\rho }^{\prime },\boldsymbol{\rho }^{\prime }\rangle = 1,
\end{equation*}
where the differentiation with respect to $s $ is denoted by a prime and $%
\langle \,,\rangle$ denotes the standard scalar product in $\mathbb{E}^{3}$.
It is easily verified that the first and the second fundamental forms of $S$
are given by
\begin{align*}
I &= n \, ds^{2} + dt^{2}, \\
II &= \frac{m}{\sqrt{n}}\,ds^{2} + \frac{2A}{\sqrt{n}}\,ds\,dt,
\end{align*}%
where
\begin{align*}
n &= \langle \boldsymbol{\sigma }^{\prime },\boldsymbol{\sigma }%
^{\prime}\rangle +2\langle \boldsymbol{\sigma }^{\prime },\boldsymbol{\rho }
^{\prime }\rangle t+t^{2}, \\
m &= \left( \boldsymbol{\sigma }^{\prime },\boldsymbol{\rho },\boldsymbol{%
\sigma }^{\prime \prime }\right) +\left[ \left( \boldsymbol{\sigma }%
^{\prime},\boldsymbol{\rho },\boldsymbol{\rho }^{\prime \prime }\right)
+\left(\boldsymbol{\rho }^{\prime },\boldsymbol{\rho },\boldsymbol{\sigma }%
^{\prime \prime }\right) \right] t+\left( \boldsymbol{\rho }^{\prime },%
\boldsymbol{\rho },\boldsymbol{\rho }^{\prime \prime }\right) t^{2}, \\
A &= \left( \boldsymbol{\sigma }^{\prime },\boldsymbol{\rho },\boldsymbol{%
\rho }^{\prime }\right).
\end{align*}
If, for simplicity, we put
\begin{align*}
\zeta &:= \langle \boldsymbol{\sigma }^{\prime },\boldsymbol{\sigma }%
^{\prime }\rangle ,\qquad \eta :=\langle \boldsymbol{\sigma } ^{\prime },%
\boldsymbol{\rho }^{\prime }\rangle , \\
\mu &:=\left( \boldsymbol{\rho }^{\prime },\boldsymbol{\rho },\boldsymbol{%
\rho }^{\prime \prime }\right),\quad \nu :=\left( \boldsymbol{\sigma }%
^{\prime },\boldsymbol{\rho },\boldsymbol{\rho }^{\prime \prime
}\right)+\left( \boldsymbol{\rho }^{\prime },\boldsymbol{\rho },\boldsymbol{%
\sigma}^{\prime \prime }\right),\quad \xi :=\left( \boldsymbol{\sigma }%
^{\prime },\boldsymbol{\rho },\boldsymbol{\sigma }^{\prime \prime }\right),
\end{align*}%
we have
\begin{equation*}
n = t^{2} + 2\eta \, t + \zeta , \quad m = \mu \, t^{2} + \nu \,t + \xi .
\end{equation*}
For the Gauss curvature $K$ of $S$ we find
\begin{equation*}
K = -\frac{A^{2}}{n^{2}}.
\end{equation*}
The second Beltrami differential operator with respect to the third
fundamental form is defined by
\begin{equation*}
\triangle^{III}f = \frac{-1}{\sqrt{e}} \, \frac{\partial \left(\sqrt{e} \,
e^{ij} \frac{\partial f}{\partial u^i}\right)}{\partial u^j},
\end{equation*}
where $f$ is a sufficient differentiable function on $S$ and $e : =
\det(e_{ij})$. After a long computation it can be expressed as follows:
\begin{eqnarray}
\triangle ^{III} & = &-\frac{n}{A^{2}}\frac{\partial ^{2}}{\partial s^{2}}+%
\frac{2n \,m}{A^{3}}\frac{\partial ^{2}}{\partial s\partial t}-\left( \frac{%
n^{2}}{ A^{2}}+\frac{n\,m^{2}}{A^{4}}\right) \frac{\partial ^{2}}{\partial
t^{2}}  \notag \\
&&+\left( \frac{n_{s}}{2A^{2}}+\frac{n \,m_{t}}{A^{3}}-\frac{m\,n_{t}}{2A^{3}%
}\right) \frac{\partial }{\partial s}  \notag \\
&&+\left( \frac{n\,m_{s}}{A^{3}}-\frac{m\,n_{s}}{2A^{3}}-\frac{m\,nA^{\prime
}}{A^{4}}-\frac{n\,n_{t}}{2A^{2}}+\frac{m^{2}\,n_{t}}{2A^{4}}-\frac{%
2n\,m\,m_{t}}{A^{4}} \right) \frac{\partial }{\partial t}  \notag \\
&=&P_{1}\frac{\partial ^{2}}{\partial s^{2}}+P_{2}\frac{\partial ^{2}}{%
\partial s\partial t}+P_{3}\frac{\partial }{\partial s}+P_{4}\frac{\partial
}{\partial t}+P_{5}\frac{\partial ^{2}}{\partial t^{2}},  \label{e4}
\end{eqnarray}%
where
\begin{equation*}
n_{t}=\frac{\partial n}{\partial t}, \quad n_{s}=\frac{\partial n}{\partial s%
},\quad m_{t}=\frac{\partial m}{\partial t},\quad m_{s}=\frac{\partial m}{%
\partial s}
\end{equation*}
and $P_{1},\dotsc,P_{5}\ $ are polynomials in $t$ with functions in $s$ as
coefficients and $\deg(P_{i})\leq 6$. More precisely we have
\begin{eqnarray*}
P_{1} &=&-\frac{1}{A^{2}}\big[t^{2}+2\eta \, t+\zeta \big], \\
P_{2} &=&\frac{2}{A^{3}}\big[\mu \, t^{4}+\left( 2\eta \, \mu +\nu
\right)t^{3}+\left( 2\eta \, \nu +\xi +\zeta \, \mu \right) t^{2}+\left(
2\eta\,\xi +\zeta \,\nu \right) t+\zeta \,\xi \big], \\
P_{3} &=&\frac{1}{A^{3}}\big[\mu \,t^{3}+3\eta \,\mu \,t^{2}+\left( \eta
\,\nu - \xi +2\zeta \,\mu + \eta ^{\prime }\, A\right) t + \frac{1}{2}\zeta
^{\prime }\,A-\eta \,\xi +\zeta \, \nu \big], \\
P_{4} &=&\frac{1}{A^{4}}\big[-3\mu ^{2}t^{5}+\left( \mu ^{\prime }\,A-\mu
\,A^{\prime 2}\right) t^{4} \\
&&+\big(\nu ^{\prime }A-\nu \,A^{\prime }+2\eta \,\mu ^{\prime }A-2\eta
\,\mu A^{\prime }-\eta ^{\prime }\mu\, A \\
&&-A^{2}-10\eta \,\mu \,\nu -2\mu \,\xi -\nu ^{2}-4\zeta \,\mu ^{2}\big)t^{3}
\\
&&+\big(\zeta\, \mu ^{\prime }A-\zeta\, \mu A^{\prime }-\frac{1}{2}\zeta
^{\prime }\,\mu \,A+2\eta\, \nu ^{\prime }A-2\eta \,\nu\, A^{\prime }-\eta
^{\prime }\,\nu A \\
&&-\xi A^{\prime }+\xi ^{\prime 2}-3\eta\, \nu^{2}-6\eta\, \mu \, \xi
-6\zeta\, \mu\, \nu \big)t^{2} \\
&&+\big(\zeta\, \nu ^{\prime }A-\zeta\, \nu\, A^{\prime }-\frac{1}{2}\zeta
^{\prime \prime }A-2\eta\, \xi \,A^{\prime }-\eta ^{\prime }\,\xi \,A \\
&&-\zeta A^{2}-2\eta ^{2}A^{2}-2\zeta \,\nu ^{2}+\xi ^{2}-2\eta\, \nu\,\xi
-4\zeta \,\mu \,\xi \big)t \\
&&+\big(\zeta\, \xi ^{\prime }A-\zeta\, \xi \,A^{\prime }-\frac{1}{2}%
\zeta^{\prime }\xi \,A+\eta \,\xi ^{2}-\zeta\, \eta\, A^{2}-2\zeta\, \nu\,
\xi \big)\big], \\
P_{5} &=&-\frac{1}{A^{4}}\big[\mu ^{2}t^{6}+\left( 2\mu\, \nu +2\eta\,
\mu^{2}\right) t^{5}+\left(2\mu\, \xi +\nu ^{2}+4\eta\, \mu\, \nu +\zeta\,
\mu ^{2}+A^{2}\right)t^{4} \\
&&+\left(2\nu\, \xi +4\eta \,\mu \,\xi +2\eta \,\nu ^{2}+2\zeta \,\mu\,
\nu+4\eta\, A^{2}\right)t^{3} \\
&&+\left(\xi ^{2}+4\eta \,\nu\, \xi +2\zeta \,\mu \,\xi +\zeta \,\nu
^{2}+4\eta^{2}A^{2}+2\zeta A^{2}\right)t^{2} \\
&&+\left(2\eta \,\xi ^{2}+2\zeta\, \nu\, \xi +4\eta\, \zeta
\,A^{2}\right)t+\left(\zeta\,\xi ^{2}+\zeta ^{2}A^{2}\right)\big].
\end{eqnarray*}
Applying (\ref{e4}) on the position vector \eqref{2.1} of the ruled surface $%
S$ we find
\begin{equation*}
\triangle ^{III}\boldsymbol{x} = P_{1}\boldsymbol{\sigma }^{\prime \prime }+
P_{2}\boldsymbol{\rho }^{\prime }+ P_{3}\boldsymbol{\sigma }^{\prime }+
P_{4} \boldsymbol{\rho } + (P_{1}\boldsymbol{\rho }^{\prime \prime }+ P_{3}%
\boldsymbol{\rho }^{\prime }) t.
\end{equation*}%
We write this last expression of $\triangle ^{III}\boldsymbol{x}$ as a
vector $\boldsymbol{Q}_{1}(t)$ whose components are polynomials in $t$ with
functions in $s$ as coefficients as follows:
\begin{eqnarray*}
\boldsymbol{Q}_{1}(t) &=&\frac{1}{A^{4}}\Big[-3\mu ^{2}\boldsymbol{\rho }%
t^{5}+\Big(\left( \mu ^{\prime }A-\mu A^{\prime 2}\right) \boldsymbol{\rho }%
+3\mu A\boldsymbol{\rho }^{\prime }\Big)t^{4} \\
&&+\Big(\mu A\boldsymbol{\sigma }^{\prime 2}\boldsymbol{\rho }^{\prime
\prime }+\big(2\nu A+7\eta \mu A\big)\boldsymbol{\rho }^{\prime } \\
&&+\big(\nu ^{\prime }A-\nu A^{\prime }+2\eta \mu ^{\prime }A-2\eta \mu
A^{\prime }-\eta ^{\prime }\mu A \\
&&-A^{2}-10\eta \mu \nu -2\mu \xi -\nu ^{2}-4\zeta \mu ^{2}\big)\boldsymbol{%
\rho }\Big)t^{3} \\
&&+\Big(\big(\zeta \mu ^{\prime }A-\zeta \mu A^{\prime }-\frac{1}{2}\zeta
^{\prime }\mu A+2\eta \nu ^{\prime }A-2\eta \nu A^{\prime }-\eta^{\prime
}\nu A \\
&&-\xi A^{\prime }+\xi ^{\prime 2}-3\eta \nu^{2}-6\eta \mu \xi -6\zeta \mu
\nu \big)\boldsymbol{\rho } \\
&&+3\eta \mu A\boldsymbol{\sigma }^{\prime 2}\boldsymbol{\rho }^{\prime
\prime }-A^{2}\boldsymbol{\sigma }^{\prime \prime }+\big(\eta ^{\prime
}A+5\eta \nu +4\zeta \mu +\xi \big)A\boldsymbol{\rho }^{\prime }\Big)t^{2} \\
&&+\Big(\big(\zeta \nu ^{\prime }A-\zeta \nu A^{\prime }-\frac{1}{2}%
\zeta^{\prime }\nu A+2\eta \xi ^{\prime }A-2\eta \xi A^{\prime }-\eta
^{\prime }\xi A \\
&&-\zeta A^{2}-2\eta ^{2}A^{2}-2\zeta \nu ^{2}+\xi ^{2}-2\eta \nu\xi -4\zeta
\mu \xi \big)\boldsymbol{\rho }-2\eta A^{2}\boldsymbol{\sigma }^{\prime
\prime } \\
&&+\big(\frac{1}{2}\zeta ^{\prime }A+3\zeta \nu +3\eta \xi \big)A\boldsymbol{%
\rho }^{\prime 2}\boldsymbol{\rho }^{\prime \prime }+\left( \eta \nu -\xi
+2\zeta \mu +\eta ^{\prime }A\right) A\boldsymbol{\sigma }^{\prime }\Big)t \\
&&+(\zeta \xi ^{\prime }A-\zeta \xi A^{\prime }-\frac{1}{2}\zeta^{\prime
2}-\zeta \eta A^{2}-2\zeta \nu \xi ) \boldsymbol{\rho } \\
&&+\big(\frac{1}{2}\zeta ^{\prime }A-\eta \xi +\zeta \nu \big)A\boldsymbol{%
\sigma }^{\prime }+2\zeta \xi A\boldsymbol{\rho }^{\prime 2} \boldsymbol{%
\sigma }^{\prime \prime }\Big].
\end{eqnarray*}
Notice that $\deg(\boldsymbol{Q}_{1})\leq 5$. Furthermore $\deg(\boldsymbol{Q%
}_{1})=5$ if and only if $\mu \neq 0$, otherwise $\deg(\boldsymbol{Q}%
_{1})\leq3$.

Before we start the proof of the first theorem we give the following Lemma
which can be proved by a straightforward computation.

\begin{lemma}
\label{L2.1} Let $g$ be a polynomial in $t$ with functions in $s$ as
coefficients and $\deg(g) = d$. Then $\triangle ^{III}g = \widehat{g}$,
where $\widehat{g}$ is a polynomial in $t$ with functions in $s$ as
coefficients and $\deg(\widehat{g})\leq d + 4$.
\end{lemma}

We suppose that $S$ is of finite $III$-type $k$. It is well known that there
exist real numbers $c_{1},\dotsc,c_{k}$ such that
\begin{equation}  \label{e5}
\left(\triangle ^{III}\right)^{k+1}\,\boldsymbol{x} + c_{1}\left(\triangle
^{III}\right)^{k} \, \boldsymbol{x} + \dotsb + c_{k}\, \triangle ^{III}\,%
\boldsymbol{x} = \mathbf{0},
\end{equation}%
see \cite{R3}. By applying Lemma \ref{L2.1}, we conclude that there is an $%
\mathbb{E}^{3}$-vector-valued function $\boldsymbol{Q}_{k}$ in the variable $%
t$ with some functions in $s$ as coefficients, such that
\begin{equation*}
\left(\triangle ^{III}\right)^{k} \, \boldsymbol{x} = \boldsymbol{Q}_{k}(t),
\end{equation*}%
where $\deg(\boldsymbol{Q}_{k})\leq 4k+1$. Now, if $k$ goes up by one, the
degree of each component of $\boldsymbol{Q}_{k}$ goes up at most by 4. Hence
the sum (\ref{e5}) can never be zero, unless of course
\begin{equation}  \label{tr}
\triangle ^{III}\,\boldsymbol{x} = \boldsymbol{Q}_{1}=\mathbf{0}.
\end{equation}
On account of the well known relation
\begin{equation*}
\triangle^{III} \, \boldsymbol{x} = \nabla^{III}\left(\frac{2H}{K},%
\boldsymbol{n}\right) -\frac{2H}{K} \, \boldsymbol{n},
\end{equation*}
where $H, \boldsymbol{n}$ and $\nabla^{III}$ denote the mean curvature, the
unit normal vector field and the first Beltrami-operator with respect to $%
III $, see \cite{R13}, from \eqref{tr} we result that $S$ is minimal, and
that $S $ is a helicoid.

\section{Proof of Theorem \protect\ref{T1.2}}

Let now $S$ be a quadric in $\mathbb{E}^{3}$. Then $S$ is either a ruled surface or one of the following two kinds, see \cite{R7},
\begin{equation}  \label{I}
z^{2} - a \,x^{2} - b\,y^{2} = c, \quad a,b,c \in \mathbb{R}, \quad a\,b
\neq 0, \quad c > 0,
\end{equation}
or
\begin{equation}  \label{II}
z = \frac{a}{2}\,x^{2} + \frac{b}{2} \,y^{2}, \quad a,b \in \mathbb{R},
\quad a,b > 0.
\end{equation}
If $S$ is a ruled surface of finite $III$-type, then, according to theorem %
\ref{T1.1}, $S$ is a helicoid.

In this section we will first show that a quadric of the kind (\ref{I}) is of finite $III$-type if and only if $a = -1$ and $b = -1$, that is, if and
only if $S$ is a sphere. Next we will show that a quadric of the kind (\ref{II}) is of infinite type.

\subsection{Quadrics of the first kind}

A part of a quadric of this kind can be parametrized by
\begin{equation}  \label{3.10}
\boldsymbol{x}(u,v) = \left( u,v,\sqrt{c + a \,u^{2} + b \, v^{2}}\right),
\quad \sqrt{c + a \,u^{2} + b \, v^{2}} > 0.
\end{equation}
We put for simplicity
\begin{equation*}
c + a\,u^{2} + b\,v^{2} \colon = \omega.
\end{equation*}
The 
third fundamental form of $S$ becomes
\begin{align*}
III &= \frac{a^{2}}{\omega \, T^{2}}C(u,v)du^{2} - 2\frac{a\,b}{\omega \,
T^{2}} B(u,v) du \,dv +\frac{b^{2}}{\omega \,T^{2}}A(u,v)dv^{2},
\end{align*}
where
\begin{align*}
T &= c + a(a+1)u^{2} + b(b+1)v^{2}, \\
A(u,v)& = a^2\,u^2\,v^2 + (a\,u^{2} + c)^{2} + a^{2} u^{2}\omega, \\
B(u,v) &= u\,v\left[ c(a + b) + a\,b(u^{2} + v^{2} + \omega )\right], \\
C(u,v) &= b^2\,u^2 \,v^2 + (b\,v^{2} + c)^{2} + b^{2} v^{2}\omega.
\end{align*}
Then the second Beltrami operator $\triangle^{III}$ of $S$ can be expressed
as follows:
\begin{eqnarray}
\triangle ^{III} &=&-\frac{T}{ a^2\,b^2\,c^2}\Bigg[ b^{2}A\frac{\partial ^{2}%
} {\partial u^{2}} + 2a \,b\,B \frac{\partial ^{2}}{\partial u\partial v} +
a^{2} C \frac{\partial ^{2}} {\partial v^{2}}\Bigg]  \notag \\
&&-\frac{T}{ a^2\,b^2\,c^2}\Bigg[ b\left( b\frac{\partial A}{\partial u} + a%
\frac{\partial B} {\partial v}\right) \frac{\partial }{\partial u} + a\left(
a \frac{\partial C} {\partial v} + b\frac{\partial B} {\partial u}\right)
\frac{\partial } {\partial v}\Bigg]  \notag \\
&& + \frac{T}{ a^2\,b^2\,c^2}\Bigg[ \frac{a\,b^{2}}{\omega }\left(u\,A +
v\,B\right) \frac{\partial }{\partial u} + \frac{a^{2}\,b}{\omega }\left( u
\,B + v\,C \right) \frac{\partial } {\partial v}\Bigg]  \notag \\
&& + \frac{1}{ a^2\,b^2\,c^2}\Bigg[a\,b^{2}\left( \left( a + 1\right) u\,A +
\left(b + 1\right) v\,B\right) \frac{\partial }{\partial u}  \notag \\
&& \qquad \qquad + a^{2} b \left( \left( b + 1\right)v\,C +\left( a +
1\right) u\,B\right) \frac{\partial }{\partial v}\Bigg].  \label{e6}
\end{eqnarray}
We note that
\begin{align*}
b\frac{\partial A}{\partial u} + a\frac{\partial B}{\partial v} &= a\,u %
\left[5a\,b(a+1)u^{2}+5a\,b(b+1)v^{2} + c (3a\,b +5 b + a)\right], \\
a\frac{\partial C}{\partial v}+b\frac{\partial B}{\partial u}&= b\, v \left[%
5a\,b(a+1)u^{2} + 5a\,b(b+1)v^{2}+c(3a\,b + 5a + b)\right], \\
u\,A + v\,B &= \left[ c + a(a+1)u^{2} + a (b+1)v^{2}\right] u\omega, \\
u\,B + v\,C &= \left[ c + b(a+1)u^{2} + b(b+1)v^{2}\right] v\omega, \\
\left( a +1 \right) u\,A +\left( b + 1\right) v\,B &= \left[%
c(a+1)+a(a+1)u^{2}+a(b+1)v^{2}\right] u\,T, \\
\left( b+1\right) v\,C+\left( a+1\right) u\,B &= \left[%
c(b+1)+b(a+1)u^{2}+b(b+1)v^{2}\right] v\,T.
\end{align*}
Hence (\ref{e6}) becomes
\begin{eqnarray}
\triangle ^{III} & = &-\frac{a(a+1)^{2}u^{5}}{c^{2}}\left( u\frac{\partial
^{2}}{\partial u^{2}}+3\frac{\partial }{\partial u}\right) -\frac{%
b(b+1)^{2}v^{5}}{ c^{2}} \left( v\frac{\partial ^{2}}{\partial v^{2}}+3\frac{%
\partial }{\partial v}\right)  \notag \\
&&+f_{1}(u,v)\frac{\partial ^{2}}{\partial u\partial v}+f_{2}(u,v)\frac{%
\partial ^{2}}{\partial u^{2}}+f_{3}(u,v)\frac{\partial ^{2}}{\partial v^{2}}
\notag \\
&&+f_{4}(u,v)\frac{\partial }{\partial u}+f_{5}(u,v)\frac{\partial }{%
\partial v},  \label{e7}
\end{eqnarray}%
where
\begin{eqnarray*}
f_{1}(u,v) &=&-2uv\left( \frac{a(a+1)^{2}}{c^{2}}u^{4}+\frac{(a+1)(a+ab+2b)}{%
bc}u^{2}+\frac{a+b+ab}{ab}\right) \\
&&-2uv\left( \frac{b(b+1)^{2}}{c^{2}}v^{4}+\frac{(b+1)(b+ab+2a)}{ac}%
v^{2}\right) \\
&&-2uv\left( \frac{(a+1)(b+1)(a+b)}{c^{2}}u^{2}v^{2}\right), \\
f_{2}(u,v) &=&-\frac{(a+1)(a+3)}{c}u^{4}-\frac{(2a+3)}{a}u^{2}-\frac{c}{a^{2}%
} \\
&&-\frac{(a+1)(b+1)(a+b)}{c^{2}}u^{4}v^{2}-\frac{b(b+1)^{2}}{c^{2}}u^{2}v^{4}
\\
&&-\frac{(b+1)(a+ab+2b)}{ac}u^{2}v^{2}-\frac{b(b+1)}{a^{2}}v^{2}, \\
f_{3}(u,v) &=&-\frac{(b+1)(b+3)}{c}v^{4}-\frac{(2b+3)}{b}v^{2}-\frac{c}{b^{2}%
} \\
&&-\frac{(a+1)(b+1)(a+b)}{c^{2}}u^{2}v^{4}-\frac{a(a+1)^{2}}{c^{2}}u^{4}v^{2}
\\
&&-\frac{(a+1)(2a+ab+b)}{bc}u^{2}v^{2}-\frac{a(a+1)}{b^{2}}u^{2}, \\
f_{4}(u,v) &=&-\frac{(a+1)(a+6b+2ab)}{bc}u^{3}-\frac{(2ab+a+3b)}{ab}u \\
&&-\frac{3(a+1)(b+1)(a+b)}{c^{2}}u^{3}v^{2}-\frac{3b(b+1)^{2}}{c^{2}}uv^{4}
\\
&&-\frac{(b+1)(4a+2ab+3b)}{ac}uv^{2}, \\
f_{5}(u,v) &=&-\frac{(b+1)(6a+b+2ab)}{ac}v^{3}-\frac{(2ab+3a+b)}{ab}v \\
&&-\frac{3(a+1)(b+1)(a+b)}{c^{2}}u^{2}v^{3}-\frac{3a(a+1)^{2}}{c^{2}}u^{4}v
\\
&&-\frac{(a+1)(3a+2ab+4b)}{bc}u^{2}v.
\end{eqnarray*}
Here again the functions $f_{i},i = 1,\dotsc,5$, are polynomials in $u$ and $%
v$ with $\deg (f_{i})\leq 6$.

We consider a function $g(u)\in C^{\infty }(U)$. By means of (\ref{e7}), we
find
\begin{equation}  \label{e8}
\triangle ^{III}g = -\frac{a(a+1)^{2}u^{5}}{c^{2}}\left( u\frac{\partial
^{2}g}{\partial u^{2}}+3\frac{\partial g}{\partial u}\right) +f_{2}(u,v)%
\frac{\partial ^{2}g} {\partial u^{2}} + f_{4}(u,v)\frac{\partial g}{%
\partial u}.
\end{equation}
If we put $v=0$, then the functions $f_{2}$ and $f_{4}$ are polynomials in $%
u $ of degree less than or equal $4$. Now we prove the following

\begin{lemma}
\label{L1} The relation
\begin{equation*}
\left( \triangle ^{III}\right) ^{k}u = \left( -1\right) ^{k} \left( \overset{%
2k}{\underset{i = 1}{\prod }}\left( 2i-1\right) \right) \left( \dfrac{
a^{k}\left( a+1\right) ^{2k}u^{4k+1}}{c^{2k}}\right) +P_{4k}(u,v),
\end{equation*}
holds true, where $\deg(P_{4k}(u,0))\leq 4k$.
\end{lemma}

\begin{proof}
The proof goes by induction on $k$. For $k=1$ the formula follows
immediately from (\ref{e8}) applied to $g = u$. Suppose the Lemma is true
for $k-1$. Then
\begin{equation*}
\left( \triangle ^{III}\right) ^{k-1}u=\left( -1\right) ^{k-1}\left( \overset%
{2k-2}{\underset{i=1}{\prod }}\left( 2i-1\right) \right) \left( \frac{
a^{k-1}\left( a+1\right) ^{2k-2}u^{4k-3}}{c^{2k-2}}\right) +P_{4k-4}(u,v).
\end{equation*}
Taking into account (\ref{e8}) we obtain
\begin{eqnarray*}
\left( \triangle ^{III}\right) ^{k}u &=&\triangle ^{III}\Big(\left(
\triangle^{III}\right) ^{k-1}u\Big)= -\frac{a(a+1)^{2}u^{5}}{c^{2}}\left(
-1\right) ^{k-1} \left( \overset{2k-2}{\underset{i=1}{\prod }}\left(
2i-1\right) \right) \\
&&\left( \frac{a^{k-1}\left(a+1\right) ^{2k-2}}{c^{2k-2}}\right) \left( u%
\frac{\partial ^{2}}{\partial u^{2}}\left( u^{4k-3}\right) +3\frac{\partial
}{\partial u}\left( u^{4k-3}\right) \right) \\
&&-\frac{a(a+1)^{2}u^{5}}{c^{2}}\left( u\frac{\partial ^{2}}{\partial u^{2}}%
\left( P_{4k-4}\right) +3\frac{\partial }{\partial u}\left( P_{4k-4}\right)
\right) \\
&&-\left( -1\right) ^{k}\left( \overset{2k-2}{\underset{i=1}{\prod }}\left(
2i-1\right) \right) \left( \frac{a^{k-1}\left( a+1\right) ^{2k-2}}{c^{2k-2}}%
\right) f_{2}(u,v)\frac{\partial ^{2}}{\partial u^{2}}\left( u^{4k-3}\right)
\\
&&-\left( -1\right) ^{k}\left( \overset{2k-2}{\underset{i=1}{\prod }}%
\left(2i-1\right) \right) \left( \frac{a^{k-1}\left( a+1\right) ^{2k-2}}{%
c^{2k-2}}\right) f_{4}(u,v)\frac{\partial }{\partial u}\left( u^{4k-3}\right)
\\
&&+f_{2}(u,v)\frac{\partial ^{2}}{\partial u^{2}}\left(
P_{4k-4}\right)+f_{4}(u,v)\frac{\partial }{\partial u}\left( P_{4k-4}\right)
\\
&=&\left( -1\right) ^{k}\left( \overset{2k}{\underset{i=1}{\prod }}\left(
2i-1\right) \right) \left( \frac{a^{k}\left( a+1\right) ^{2k}u^{4k+1}}{c^{2k}%
}\right) +P_{4k}(u,v),
\end{eqnarray*}
where
\begin{eqnarray}
P_{4k}(u,v) &=&-\frac{a(a+1)^{2}u^{5}}{c^{2}}\left( u\frac{\partial ^{2}}{
\partial u^{2}}\left( P_{4k-4}\right) +3\frac{\partial }{\partial u}\left(
P_{4k-4}\right) \right)  \notag \\
&&-\left( -1\right) ^{k}\left( \overset{2k-2}{\underset{i=1}{\prod }}\left(
2i-1\right) \right) \left( \frac{a^{k-1}\left( a+1\right) ^{2k-2}}{c^{2k-2}}
\right) f_{2}(u,v)\frac{\partial ^{2}}{\partial u^{2}}\left( u^{4k-3}\right)
\notag \\
&&-\left( -1\right) ^{k}\left( \overset{2k-2}{\underset{i=1}{\prod }}\left(
2i-1\right) \right) \left( \frac{a^{k-1}\left( a+1\right) ^{2k-2}}{c^{2k-2}}
\right) f_{4}(u,v)\frac{\partial }{\partial u}\left( u^{4k-3}\right)  \notag
\\
&&+f_{2}(u,v)\frac{\partial ^{2}}{\partial u^{2}}\left( P_{4k-4}\right)
+f_{4}(u,v)\frac{\partial }{\partial u}\left( P_{4k-4}\right) .  \label{e9}
\end{eqnarray}
Since
\begin{equation*}
\deg \left( P_{4k-4}(u,0)\right) \leq 4k-4, \quad \deg (f_{2}(u,0))\leq 4
\quad \text{and} \quad \deg (f_{4}(u,0))\leq 4,
\end{equation*}
from (\ref{e9}) we find that $\deg(P_{4k}(u,0))\leq 4k$.
\end{proof}

By applying now (\ref{e7}) on a function $h(v)\in C^{\infty }(U)$ we find
\begin{equation*}
\triangle ^{III}h=-\frac{b(b+1)^{2}v^{5}}{c^{2}}\left( v\frac{\partial ^{2}h%
}{\partial v^{2}}+3\frac{\partial h}{\partial v}\right) +f_{3}(u,v)\frac{
\partial ^{2}h}{\partial v^{2}}+f_{5}(u,v)\frac{\partial h}{\partial v}.
\end{equation*}
If we put $u = 0$, then the functions $f_{3}$ and $f_{5}$ are polynomials in
$v$ of degree less than or equal $4$. Proceeding analogously as in Lemma \ref%
{L1}, we prove the following

\begin{lemma}
\label{L2} The relation
\begin{equation*}
\left( \triangle^{III}\right)^{k}v = \left(-1\right)^{k} \left(\overset{2k}{%
\underset{i=1}{\prod }}\left(2i-1\right)\right)\left(\dfrac{ b^{k}
\left(b+1\right)^{2k}v^{4k+1}}{c^{2k}}\right)+Q_{4k}(u,v)
\end{equation*}
holds true, where $\deg (Q_{4k}(0,v))\leq 4k $.
\end{lemma}

We suppose now that $S$ is of finite $III$-type $k$. Then there exist real
numbers $c_{1},\dotsc,c_{k}$ such that
\begin{equation}  \label{e11}
\left(\triangle ^{III}\right)^{k+1} \, \boldsymbol{x} + c_{1}\,
\left(\triangle ^{III} \right)^{k} \, \boldsymbol{x} + \dotsc +
c_{k}\triangle ^{III} \, \boldsymbol{x} = \mathbf{0}.
\end{equation}
Applying (\ref{e11}) on the coordinate functions $x_{1} = u$ and $x_{2} = v$
of the position vector \eqref{3.10} of the quadric $S$ we obtain
\begin{align}
\left(\triangle ^{III}\right)^{k+1} \, u + c_{1}\left( \triangle
^{III}\right )^{k} \, u + \dotsb + c_{k}\triangle ^{III} u = 0,  \label{e12}
\\
\left(\triangle ^{III}\right)^{k+1}v + c_{1}\left(\triangle
^{III}\right)^{k} v + \dotsb + c_{k}\triangle ^{III}v = 0.  \label{e13}
\end{align}
From Lemma \ref{L1} and the relation (\ref{e12}) we obtain that there exists
a polynomial $P_{4k+4}(u,v)$ of degree at most $4k+4$ such that
\begin{equation}  \label{e14}
\left( -1\right) ^{k+1}\left( \overset{2k+2}{\underset{i=1}{\prod }}\left(
2i-1\right) \right) \left( \frac{a^{k+1}\left( a+1\right) ^{2k+2}u^{4k+5}}{%
c^{2k+2}}\right) +P_{4k+4}(u,v)=0.
\end{equation}
If we put $v=0$ in (\ref{e14}), then we get a nontrivial polynomial in $u$
with constant coefficients. Since $a\neq 0$ the relation (\ref{e14}) implies
$a = -1$.

Similarly, from Lemma \ref{L2} and the relation (\ref{e13}) we obtain that
there exists a polynomial $Q_{4k+4}(u,v)$ of degree at most $4k+4$ such that
\begin{equation}  \label{e15}
\left( -1\right) ^{k+1}\left( \overset{2k+2}{\underset{i=1}{\prod }}\left(
2i-1\right) \right) \left( \frac{b^{k+1}\left( b+1\right) ^{2k+2}v^{4k+5}}{%
c^{2k+2}}\right) +Q_{4k+4}(u,v)=0.
\end{equation}
Putting $u=0$ in (\ref{e15}), we get again a nontrivial polynomial in $v$
with constant coefficients. Since $b\neq 0$, we obtain from (\ref{e15}) $%
b=-1 $. Hence $S$ must be a sphere.

\subsection{Quadrics of the second kind}

A quadric surface of this kind can be parametrized by
\begin{equation}  \label{4.10}
\boldsymbol{x}(u,v) = \left( u,v,\frac{a}{2}u^{2} + \frac{b}{2}v^{2}\right).
\end{equation}
Then the 
third fundamental form of $S$ 
is the following
\begin{align*}
III&= \frac{a^{2}}{g^{2}}(1+b^{2}v^{2})du^{2}-2\frac{a^{2}b^{2}} {g^{2}}%
u\,v\,du\,dv+\frac{b^{2}}{g^{2}}(1+a^{2}u^{2})dv^{2},
\end{align*}
where
\begin{equation*}
g\colon = \det \left(g_{ij}\right) = 1 + a^2 \, u^2 + b^2 \, v^2
\end{equation*}
is the discriminant of the first fundamental form
\begin{equation*}
I =(1+a^2\,u^2 ) du^{2}+2a\,b\,u\,v\,du\,dv + (1+b^2\,v^2)dv^{2}
\end{equation*}
of $S$. Hence the Beltrami operator $\triangle ^{III}$ of $S$ takes the
following form
\begin{eqnarray*}
\triangle ^{III} &=& - \frac{g(1 + a^{2}u^{2})} {a^{2}} \frac{\partial ^{2}}{%
\partial u^{2}} - \frac{g(1+b^{2}v^{2})}{b^{2}}\frac{\partial ^{2}}{\partial
v^{2}} \\
&&-2u\,v\,g\frac{\partial ^{2}}{\partial u\partial v}-2u\,g\frac{\partial }{%
\partial u}-2v\,g\frac{\partial }{\partial v},
\end{eqnarray*}
which can be written as
\begin{eqnarray}
\triangle ^{III} & =&-a^{2}u^{3}\left( u\frac{\partial ^{2}}{\partial u^{2}}
+2 \frac {\partial }{\partial u}\right) -b^{2}v^{3}\left( v\frac{\partial
^{2}}{\partial v^{2}} + 2\frac{\partial }{\partial v}\right)  \notag \\
&& -f_{1}(u,v)\frac{\partial ^{2}}{\partial u\partial v} - f_{2}(u,v)\frac{%
\partial ^{2}}{\partial u^{2}}-f_{3}(u,v)\frac{\partial ^{2}}{\partial v^{2}}
\notag \\
&&-f_{4}(u,v)\frac{\partial }{\partial u}-f_{5}(u,v)\frac{\partial }{%
\partial v},  \label{e16}
\end{eqnarray}%
where
\begin{align*}
f_{1}(u,v) &=2u\,v\,g, \\
f_{2}(u,v) &=2u^{2} + b^{2}u^{2}v^{2}+\frac{1}{a^{2}}\left(1+b^{2}v^{2}%
\right), \\
f_{3}(u,v) &=2v^{2}+a^{2}u^{2}v^{2}+\frac{1}{b^{2}}\left(1+a^{2}u^{2}\right),
\\
f_{4}(u,v) &=2u\left( 1+b^{2}v^{2}\right), \\
f_{5}(u,v) &=2v\left( 1+a^{2}u^{2}\right).
\end{align*}
Notice that the functions $f_{i},i=1,\dotsc,5$, are polynomials in $u$ and $%
v $ with $\deg (f_{i})\leq 4$. By applying the operator $\triangle^{III}$ on
a function $g(u)\in C^{\infty }(U)$ we find by means of (\ref{e16})
\begin{equation}
\triangle ^{III}g=-a^{2}u^{3}\left( u\frac{\partial ^{2}g}{\partial u^{2}}+2%
\frac{\partial g}{\partial u}\right) -f_{2}(u,v)\frac{\partial ^{2}g}{
\partial u^{2}}-f_{4}(u,v)\frac{\partial g}{\partial u}.  \label{e17}
\end{equation}

If we put $v=0$, then the functions $f_{2}$ and $f_{4}$ are polynomials in $%
u $ of degree less than or equal $2.$

Using (\ref{e17}) and by induction on $k$ we can prove the following

\begin{lemma}
\label{L3} The relation
\begin{equation*}
\left(\triangle ^{III}\right)^{k}\, u =
(-1)^{k}(2k)!\,a^{2k}\,u^{2k+1}+P_{2k}(u,v),
\end{equation*}
holds true, where $\deg (P_{2k}(u,0))\leq 2k$.
\end{lemma}

By applying (\ref{e16}) on a function $h(v)\in C^{\infty }(U)$ we get
\begin{equation}
\triangle ^{III}h=-b^{2}v^{3}\left( v\frac{\partial ^{2}h}{\partial v^{2}}+2%
\frac{\partial h}{\partial v}\right) -f_{3}(u,v)\frac{\partial ^{2}h}{
\partial v^{2}}-f_{5}(u,v)\frac{\partial h}{\partial v}.  \label{e20}
\end{equation}
If we put $u=0$, then the functions $f_{3}$ and $f_{5}$ are polynomials in $%
v $ of degree less than or equal $2$. In the same way the following Lemma
can be proved

\begin{lemma}
\label{L4} The relation
\begin{equation*}
\left(\triangle ^{III}\right)^{k}\,v =
(-1)^{k}(2k)!\,b^{2k}v^{2k+1}+Q_{2k}(u,v)
\end{equation*}
holds true, where $\deg (Q_{2k}(0,v))\leq 2k$.
\end{lemma}

Now, if the quadric $S$ is of finite $III-$type $k$, then again the
relations \eqref{e11}, \eqref{e12} and \eqref{e13} are valid. Combining the
equations \eqref{e12} and \eqref{e13} with Lemma \ref{L3} and Lemma \ref{L4}
respectively, we conclude that there exist two polynomials $P_{2k+2}(u,v)$
and $Q_{2k+2}(u,v)$ of degree at most $2k+2$ such that
\begin{align}
\left( -1\right) ^{k+1}(2k+2)!\,a^{2k+2}u^{2k+3}+P_{2k+2}(u,v)=0,
\label{e22} \\
\left( -1\right) ^{k+1}(2k+2)!\,b^{2k+2}v^{2k+3}+Q_{2k+2}(u,v)=0.
\label{e23}
\end{align}
We put $v=0$ in (\ref{e22}). Then the left member of the equation (\ref{e22}%
) is a nontrivial polynomial in $u$ with constant coefficients. This
polynomial can never be zero, unless $a = 0$. Similarly, if we put $u = 0$
in (\ref{e23}), then the left member of (\ref{e23}) is a nontrivial
polynomial in $v$ with constant coefficients. This implies $b = 0$. This is
clearly impossible since $a,b > 0$. 


\end{document}